\documentclass[11pt]{amsart}
\usepackage{amsmath}
\usepackage{amscd}
\usepackage{amsthm}
\usepackage{graphicx}
\usepackage{amssymb}
\usepackage{epstopdf}
\theoremstyle{plain}
\newtheorem{Thm}{Theorem}[section]

\newtheorem{Prop}[Thm]{Proposition}
\newtheorem{Cor}[Thm]{Corollary}
\newtheorem{Lem}[Thm]{Lemma}

\theoremstyle{definition}
\newtheorem{Defn}[Thm]{Definition}
\newtheorem{Expl}[Thm]{Example}
\newtheorem{Rem}[Thm]{Remark}
 
\numberwithin{equation}{section}

\title{On non-commutative formal deformations of coherent sheaves on an algebraic variety}
\author{Yujiro Kawamata}
\usepackage{fancyhdr}
\begin{document}
\maketitle

\tableofcontents

\begin{abstract}
We review the theory of non-commutative deformations of sheaves and describe a versal deformation by using an $A^{\infty}$-algebra and the change of 
differentials of an injective resolution. 
We give some explicit non-trivial examples.

14D15, 14F05
\end{abstract}

%%%%%%%%%%%%%%%%%%%%%%%%%%%%%%%%%%%
%%%%%%%%%%%%%%%%%%%%%%%%%%%%%%%%%%%
%%%%%%%%%%%%%%%%%%%%%%%%%%%%%%%%%%%
\section{Introduction}

We consider non-commutative deformations of sheaves on an algebraic variety in this paper.
We consider also multi-pointed deformations, and give some non-trivial examples.
The point is that such deformation theory is more natural than the commutative ones as long
as we consider infinitesimal deformations.

Let $F$ be a coherent sheaf on an algebraic variety $X$ defined over a field $k$ 
such that the support of $F$ is proper.
We can consider a moduli space $M$ which parametrizes flat deformations of $F$.
The infinitesimal study of $M$ is to investigate the completed local ring $\hat {\mathcal{O}}_{M,[F]}$ at 
a point corresponding to $F$.
The tangent space of $M$ at $[F]$ is isomorphic to $\text{Ext}^1(F,F)$, and the singularitiy at $[F]$
is described by using the obstruction space $\text{Ext}^2(F,F)$.
Thus we can write
\[
\hat {\mathcal{O}}_{M,[F]} = k[[\text{Ext}^1(F,F)^*]]/(\text{Ext}^2(F,F)^*)
\]
where ${}^*$ denotes the dual vector space, $k[[\text{Ext}^1(F,F)^*]]$ 
is the completed symmetric tensor algebra 
of $\text{Ext}^1(F,F)^*$ and the denominator is a certain ideal determined by $\text{Ext}^2(F,F)^*$, 
an ideal generated by power series on a basis of $\text{Ext}^1(F,F)^*$ corresponding to the members 
of a basis of $\text{Ext}^2(F,F)^*$.

But it is more natural to consider the completed (non-symmetric) 
tensor algebra. 
We obtain the {\em  non-commutative (NC) deformation algebra}, the parameter algebra of a {\em versal NC deformation}
\[
\hat R = k \langle \langle \text{Ext}^1(F,F)^* \rangle \rangle/(\text{Ext}^2(F,F)^*)
\]
where $k \langle \langle \text{Ext}^1(F,F)^* \rangle \rangle$ is the completed tensor algebra 
\[
\begin{split}
&\hat T_k^{\bullet}\text{Ext}^1(F,F)^* = \prod_{i=0}^{\infty} (\text{Ext}^1(F,F)^*)^{\otimes i} \\
&= k \times \text{Ext}^1(F,F)^* \times (\text{Ext}^1(F,F)^* \otimes \text{Ext}^1(F,F)^*) \times \dots
\end{split}
\]
and the denominator is a certain two sided ideal determined by $\text{Ext}^2(F,F)^*$.

\vskip 1pc

The abstract existence of a versal (formal) NC deformation is proved in the same way as in
the case of commutative deformations (\cite{Schlessinger}, \cite{Laudal}).

We can describe a versal deformation, as well as proving its existence, by using $A^{\infty}$-algebra 
formalism.
Such a description is apparently well known to experts, e.g., \cite{Toda}~\S 4.
But we use injective resolutions instead of locally free resolutions.
This has advantage that our argument works not only for non-smooth non-projective varieties $X$ but also 
for objects in a $k$-linear abelian category with enough injectives.
We also put emphasis on the non-commutativity of the parameter algebras.
We treat only formal deformations, but there are results on the convergence (cf. Remarks~\ref{T} and \ref{ZH}). 

The abstract description of the versal deformation using an $A^{\infty}$-algebra does not necessarily 
give solutions to explicit deformation problems because it involves injective resolutions etc.
So we consider simple but non-trivial examples where the versal deformations are explicitly calculated. 
We prove that the versal deformation of a structure sheaf of a subvariety is described by a 
left ideal (Lemma~\ref{relation ideal}).
We apply this for lines in a projective space and prove that the relation ideal is generated by 
quadratic NC polynomials.
We also calculate the relation NC polynomials for deformations of conics and prove that they 
have degree $3$.

\vskip 1pc

The content of this paper is as follows.
In \S 2, we give a definition of non-commutative deformations of a coherent sheaf, 
and express NC deformations as a 
change of differentials in an injective resolution.
We describe them by using {\em Maurer-Cartan equation} in a differential graded associative algebra.

We review the theory of $A^{\infty}$-algebras in \S 3 in order to use it in later sections.
In \S 4, we describe a versal deformation and the deformation algebra, the parameter algebra 
of a versal deformation, 
by using a minimal model $A^{\infty}$-algebra 
of the DG-algebra considered in \S 2.
The advantage of the minimal model $A^{\infty}$ formulation is that the vector spaces are finite dimensional 
for fixed degrees, 
while the DG algebra is infinite dimensional in each degree.
In order to achieve this, we need to introduce infinitely many multi-linear maps. 
We prove the versality of the deformation constructed by using the injective resolution (Theorem~\ref{versal}). 
We extend the whole theory to its refined version of multi-pointed 
NC deformations (Theorem~\ref{r-pointed}) in \S 5.
We make some remarks on the relationship of NC deformations and 
iterated self extensions in \S 6. 

We consider some explicit non-trivial examples in \S 7.
In Example~\ref{lines} on lines in a projective space, 
we prove that the higher multiplications $m_i$ for $i \ge 3$ vanish, while in
Example~\ref{conics} on conics in $\mathbf{P}^4$, we prove that $m_3 \ne 0$ but $m_i = 0$ for $i \ge 4$.

The author would like to thank Professors Keiji Oguiso, Spela Spenko 
and Michel Van den Bergh for the 
information on the refernces \cite{GS} and \cite{Vinberg} in Remark~\ref{GS}.
The author would also like to thank Professors Yukinobu Toda and Zheng Hua 
for useful discussions
(cf. Remarks~\ref{T} and \ref{ZH}). 

%%%%%%%%%%%%%%
%%%%%%%%%%%%%%
%%%%%%%%%%%%%%
%%%%%%%%%%%%%%
\section{Non-commutative deformations and DG algebra}

We consider $1$-pointed non-commutative (NC) deformations of a coherent sheaf in this section.
The extension to multi-pointed case is treated in a later section.

Let $X$ be an algebraic variety defined over a field $k$, and  
let $F$ be a coherent sheaf on $X$ with proper support.

Let $(\text{Art}_k)$ be the category of associative $k$-algebras $R$ with a maximal two-sided ideal $M$ 
such that $R$ is a finite dimensional $k$-module, 
$R/M \cong k$, and that $M^{n+1} = 0$ for some $n$.
It follows that $R/M$ is the only simple $R$-module and any finitely generated $R$-module
is obtained as a successive extension of $R/M$.

\begin{Defn}
Let $X, F$ be as above, $R \in (\text{Art}_k)$, let $F_R$ be a left $R \otimes_k \mathcal{O}_X$-module
which is coherent as an $\mathcal{O}_X$-module, and let $\phi: R/M \otimes_R F_R \cong F$ be an isomorphism.
Then a pair $(F_R, \phi)$ is said to be 
a {\em non-commutative deformation} of $F$ over $R$, 
if $F_R$ is flat as a left $R$-module.
\end{Defn}

Unless $M = 0$ and $R = k$, we can define a two-sided ideal 
$J = M^n$ for the maximal integer $n$ such that $M^n \ne 0$.
Then we have $MJ = 0$.
If we put $R' = R/J$, then we have $\dim_k R' < \dim_k R$.
We use this fact for the purpose of inductive argument on $\dim_k R$

We will describe NC deformation by using injective resolutions.

\begin{Lem}
Let $F$ be a coherent sheaf on $X$.
Then there is an injective $O_X$-homomorphism $i: F \to I$ to an injective $O_X$-module which satisfies the 
following condition: for any deformations $F_R$ of $F$ over any $R \in (Art_k)$, there are
injective $R \otimes_k O_X$-module homomorphisms $i_R: F_R \to R \otimes_k I$ 
such that $R/M \otimes_R i_R = i$.
\end{Lem}

\begin{proof}
For any point $x \in X$, we define stalks of $I$ by $I_x = Hom_k(O_{X,x}, F_x)$.
Then $I_x$ has an $O_{X,x}$-module structure given by 
$af(b) = f(ab)$ for $a,b \in O_{X,x}$ and $f \in I_x$.
We claim that $I_x$ is an injective $O_{X,x}$-module.
Indeed, for any $O_{X,x}$-module $M$, the map 
\[
h: Hom_k(M,F_x) \to Hom_{O_x}(M,I_x)
\]
given by $h(f)(m)(a) = f(am)$ for $f \in Hom_k(M,F_x)$, $m \in M$ and $a \in O_{X,x}$, is bijective 
with inverse given by $h^{-1}(g)(m) = g(m)(1)$ for $g \in Hom_{O_x}(M,I_x)$.

There is a natural injective $O_{X,x}$-homomorphism $i_x: F_x \to I_x$ 
defined by $i_x(c)(a) = ac$ for $c \in F_x$ and $a \in O_{X,x}$.
We define an $O_X$-module $I$ by $I(U) = \prod_{x \in U} I_x$ for open subsets $U \subset X$.
Then $I$ is an injective $O_X$-module with a natural injective $O_X$-homomorphism $i: F \to I$. 

Since the stalk $F_{R,x}$ has an $R \otimes_k O_{X,x}$-module structure, the $k$-module 
$I_{R,x} = Hom_k(O_{X,x}, F_{R,x})$ has the induced $R \otimes_k O_{X,x}$-module structure
given by $raf(b) = rf(ab)$ for $a,b \in O_{X,x}$, $r \in R$ and $f \in I_{R,x}$.
We define $I_R$ by $I_R(U) = \prod_{x \in U} I_{R,x}$.
Then $I_R$ is again an injective module as an $O_X$-module, and 
there is a natural injective $R \otimes_k O_X$-homomorphism $i_R: F_R \to I_R$.

The natural surjective $O_X$-homomorphism $F_R \to F$ induces a 
surjective $R \otimes_k O_X$-homomorphism $I_R \to I$.
Since $I_R$ is $O_X$-injective, there is a splitting $O_X$-homomorphism $I \to I_R$.
By scalar extension, we obtain an $R \otimes_k O_X$-homomorphism $R \otimes_k I \to I_R$,
which is bijective due to the flatness of $F_R$ over $R$.
Therefore the lemma is proved.
\end{proof}

The above lemma is non-trivial in some sense because $R \otimes_k I$ appears in the middle
of the flow of arrows in the following diagram:
\[
\begin{CD}
F_R @>>> R \otimes_k I \\
@VVV @VVV \\
F @>>> I.
\end{CD}
\]

\begin{Cor}
There is an injective resolution
\[
0 \to F \to I^0 \to I^1 \to I^2 \to \dots
\]
as $O_X$-modules such that, for any deformation $F_R$ of $F$ over $R \in (Art_k)$, there is an 
exact sequence of $R \otimes_k O_X$-modules
\[
0 \to F_R \to R \otimes_k I^0 \to R \otimes_k I^1 \to R \otimes_k I^2 \to \dots
\]
which is reduced to the first exact sequence when the functor $R/M \otimes_R$ is applied.
\end{Cor}

\begin{proof}
We apply the lemma to the cokernels.
\end{proof}

We describe NC deformations by using differential graded (DG) associative algebras.
Let $F \to I^{\bullet}$ be an injective resolution as above, and
let 
\[
A = \text{Hom}^{\bullet}(I^{\bullet},I^{\bullet}) = \bigoplus_{i \in Z} \text{Hom}^i(I^{\bullet},I^{\bullet})
\]
be the associative DG algebra of graded
homomorphisms, where
\[
\text{Hom}^i(I^{\bullet},I^{\bullet}) = \prod_{m=0}^{\infty} \text{Hom}(I^m,I^{m+i})
\]
is the $i$-th graded piece, 
and the differential of $A$ is given by 
\[
d_Af = d_If - (-1)^ifd_I
\]
for $f \in \text{Hom}^i(I^{\bullet},I^{\bullet})$, where
$d_I$ denotes the differential of $I$.

\begin{Lem}\label{MC}
Let $(R,M) \in (\text{Art}_k)$, and let $y \in M \otimes A^1$.
Let $d_{R,I} + y$ be an endomorphism of degree $1$ of a graded 
$R \otimes_k \mathcal{O}_X$-module $R \otimes_k I^{\bullet}$, where $d_{R,I} = 1_R \otimes d_I$ denotes 
the scalar extension of $d_I$.
Then the following hold: 

(1) $(d_{R,I}+y)^2 = 0$ if and only if the {\em Maurer-Cartan equation} 
\[
d_{R,A}y + y^2 = 0
\]
is satisfied, where $d_{R,A} = 1_R \otimes d_A$ is the scalar extension of $d_A$.

(2) In this case, $\mathcal{H}^p(R \otimes_k I^{\bullet}, d_{R,I}+y) = 0$ for $p > 0$ 
and $F_R := \mathcal{H}^0(R \otimes_k I^{\bullet}, d_{R,I}+y)$ is flat over $R$. 
\end{Lem}

\begin{proof}
(1) We have $(d_{R,I}+y)(d_{R,I}+y) = d_{R,I}y+yd_{R,I} + y^2 = d_{R,A}y + y^2$.

(2) We proceed by induction on $\dim_k R$.
We take a two-sided ideal $J$ such that $MJ=0$, and let $R' = R/J$.
Then we have an exact sequence of complexes 
\[
0 \to J \otimes I^{\bullet} \to R \otimes I^{\bullet} \to R' \otimes I^{\bullet} \to 0.
\]
The associated long exact sequence yields the result.
\end{proof}

The existence of a {\em versal deformation}, or a {\em hull}, for NC deformations is proved in the same way as
in the case of commutative deformations (\cite{Schlessinger}, \cite{Laudal}).
One can describe a versal deformation using the formalism of $A^{\infty}$-algebras as explained in 
subsequent sections.

%%%%%%%%%%%%%%%%%%%%%%%%%%%%%%%%%%%%%%%
%%%%%%%%%%%%%%%%%%%%%%%%%%%%%%%%%%%
%%%%%%%%%%%%%%%%%%%%%%%%%%%%%%%%%%%
\section{Review on $A^{\infty}$-algebra}

We recall the definition of $A^{\infty}$-algebras (cf. \cite{Keller}).

\begin{Defn}
(1) Let $A = \bigoplus_{i \in \mathbf{Z}} A_i$ be a graded $k$-module.
An {\em $A^{\infty}$-algebra} structure consists of $k$-linear maps
$m_i: A^{\otimes i} \to A$ of degree $2-i$ for $i \ge 1$ satisfying the following relations:
\[
\sum_{r,t \ge 0, s \ge 1, r+s+t = i} (-1)^{rs + t} m_{r+1+t}(1^{\otimes r} \otimes m_s \otimes 1^{\otimes t}) = 0.
\]
For example, 
\[
\begin{split}
i=1: \,\,\, &m_1m_1 = 0. \\ 
i=2: \,\,\, &m_1m_2 + m_2(-m_1 \otimes 1 - 1 \otimes m_1) = 0. \\ 
i=3: \,\,\, &m_1m_3 + m_2(- m_2 \otimes 1 + 1 \otimes m_2) \\
&+ m_3(m_1 \otimes 1 \otimes 1 + 1 \otimes m_1 \otimes 1 + 1 \otimes 1 \otimes m_1) = 0.
\end{split}
\]
(2) Let $A,B$ be $A^{\infty}$-algebras.
An {\em $A^{\infty}$-algebra homomorphism} $f: A \to B$ consists of $k$-linear maps
$f_i: A^{\otimes i} \to B$ of degree $1-i$ for $i \ge 1$ satisfying the following relations:
\[
\begin{split}
&\sum_{r,t \ge 0, s \ge 1, r+s+t=i} (-1)^{rs+t}f_{r+1+t}(1^{\otimes r} \otimes m_s^A \otimes 1^{\otimes t}) \\
&= \sum_{r \ge 1, i_1+\dots+i_r = i} (-1)^{\sum_{1 \le j < k \le r} i_j(i_k+1)} 
m_r^B(f_{i_1} \otimes \dots \otimes f_{i_r}).
\end{split}
\]
For example, 
\[
\begin{split}
i=1: \,\,\, &f_1m_1^A = m_1^Bf_1. \\
i=2: \,\,\, &f_1m_2^A + f_2(-m_1^A \otimes 1 - 1 \otimes m_1^A) 
= m_1^Bf_2 + m_2^B(f_1 \otimes f_1). \\
i=3: \,\,\, &f_1m_3^A + f_2(- m_2^A \otimes 1 + 1 \otimes m_2^A) \\
&+ f_3(m_1^A \otimes 1 \otimes 1 + 1 \otimes m_1^A \otimes 1 + 1 \otimes 1 \otimes m_1^A) \\
&= m_1^Bf_3 + m_2^B(- f_1 \otimes f_2 + f_2 \otimes f_1) + m_3^B(f_1 \otimes f_1 \otimes f_1).  
\end{split}
\]
\end{Defn}

A {\em DG (differential graded) associative algebra} is a special case of an $A^{\infty}$-algebra 
where $m_1$ is the differential,
$m_2$ is the associative algebra multiplication, and $m_i = 0$ for $i \ge 3$.

Let $A$ be a DG algebra.
Then its cohomology group $H(A) = \bigoplus_i H^i(A)$ is a graded $k$-module.

\begin{Thm}[Kadeishvili \cite{Kadeishvili}]
Let $A$ be a DG associative algebra.
Then there is an $A^{\infty}$-algebra structure on the cohomology group $H(A)$ such that 
$m_1 = 0$, $m_2$ is induced from the algebra multiplication 
$m_2^A$ of $A$, and that there is 
a morphism of $A^{\infty}$-algebras $f: H(A) \to A$ such that $f_1$ lifts the identity of $H(A)$.
\end{Thm}

\begin{proof}[sketch of proof]
We define $k$-linear maps $m_n: H(A)^{\otimes n} \to H(A)$ of degree $2-n$ and 
$f_n: H(A)^{\otimes n} \to A$ of degree $1-n$ by induction on $n \ge 1$, which satisfy the 
following relations:
\[
\begin{split}
&(1) \,\,\, \sum_{r,t \ge 0, s \ge 2, r+s+t = n} (-1)^{rs + t} m_{r+1+t}(1^{\otimes r} \otimes m_s \otimes 
1^{\otimes t}) = 0. \\
&(2) \,\,\, \sum_{r,t \ge 0, s \ge 2, r+s+t=n} (-1)^{rs+t}f_{r+1+t}(1^{\otimes r} \otimes m_s \otimes 
1^{\otimes t}) \\
&= m_1^Af_n + \sum_{i=1}^{n-1}(-1)^{i(n-i+1)}m_2^A(f_i \otimes f_{n-i})
\end{split}
\]
where $m_1^A = d_A$ and $m_2^A$ is the associative multiplication.
 
First we set $m_1 = 0$.  
Let us choose $f_1: H(A) \to A$ to be any $k$-linear map which sends cohomology classes to their 
representatives.

Now assume that $m_i$ and $f_i$ are already defined for $i < n$.
Let $U_n: H(A)^{\otimes n} \to A$ be a $k$-linear map of degree $2-n$ defined by 
\[
\begin{split}
U_n = &\sum_{i=1}^{n-1}(-1)^{i(n-i+1)}m_2^A(f_i \otimes f_{n-i}) \\
&- \sum_{r,t \ge 0, 2 \le s \le n-1, r+s+t=n} (-1)^{rs+t}f_{r+1+t}(1^{\otimes r} \otimes m_s \otimes 1^{\otimes t}).
\end{split}
\]
For example, $U_2 = m_2^A(f_1 \otimes f_1)$.
Then the condition (2) becomes 
\[
m_1^Af_n + U_n = f_1m_n.
\]
A complicated calculation shows that 
$m_1^AU_n = 0$, where we need to be careful on the sign changes.

We define $m_n = [U_n]$, where $[\,\,]$ denotes the cohomology class in $H(A)$.
Then it follows that $f_1m_n - U_n \in \text{Im}(m_1^A)$.
We choose any $k$-linear map $f_n$ such that 
$m_1^Af_n = f_1m_n - U_n$, then (2) is satisfied.
Then we can check the relation (1) by a complicated calculation again.
\end{proof}

The composition of $A^{\infty}$-morphisms $f: A \to B$ and $g: B \to C$ is defined as follows:
\[
(g \circ f)_n = \sum_{r \ge 1, \sum i_j = n} (-1)^{\sum_{j < k} i_j(i_k+1)} g_r \circ (f_{i_1} \otimes \dots \otimes f_{i_r})
\]
The identity morphism $f = 1: A \to A$ is defined by $f_1 = 1$ and $f_n = 0$ for $n \ge 2$.

\begin{Prop}
Let $A$ be a DG algebra, and let $f: H(A) \to A$ be the $A^{\infty}$-morphism obtained in the previous theorem.
Then there is an $A^{\infty}$-algebra morphism
$g: A \to H(A)$ such that $g \circ f = 1_{H(A)}$.
\end{Prop}

\begin{proof}
We will define the $g_n$ inductively. 
The conditions are 
\[
\begin{split}
&\sum_{r,t \ge 0, r+1+t=n} (-1)^{r+t} g_{r+1+t}(1^{\otimes r} \otimes m^A_1 \otimes 1^{\otimes t}) \\
&+ \sum_{r,t \ge 0, r+2+t=n} (-1)^t g_{r+1+t}(1^{\otimes r} \otimes m^A_2 \otimes 1^{\otimes t}) \\
&= \sum_{r \ge 2, \sum i_j = n} (-1)^{\sum_{j<k} i_j(i_k+1)}m^{H(A)}_r(g_{i_1} \otimes \dots \otimes g_{i_r}), 
\end{split}
\]
$g_1f_1 = 1$, and
\[
\sum_{r \ge 1, \sum i_j = n} (-1)^{\sum_{j < k} i_j(i_k+1)} g_r \circ (f_{i_1} \otimes f_{i_r}) = 0
\]
for $n \ge 2$.

If the $g_i$ for $i < n$ are already determined, then $g_n$ is chosen such that it has given values on 
$f_1(A)^{\otimes n}$ and the $k$-subspace $V$ of $A^{\otimes n}$ generated by elements of the form
$x_1\otimes \dots \otimes x_r \otimes dx_{r+1} \otimes x_{r+2} \otimes \dots \otimes x_n$.
Such a $g_n$ exists because $f_1(A)^{\otimes n} \cap V = 0$.
\end{proof}

%%%%%%%%%%%%%%%%%%%%%%%%%%%
%%%%%%%%%%%%%%%%%%%%%%%%%%%
%%%%%%%%%%%%%%%%%%%%%%%%%%%
\section{Description using $A^{\infty}$-structure}

Let $F$ be a coherent sheaf on an algebraic variety $X$, and let $A = \text{Hom}^{\bullet}(I^{\bullet},I^{\bullet})$ be the DG algebra considered in \S2.
We know that $H^p(A) = \text{Ext}^p(F,F)$.
The cohomology space $H(A)$ has an $A^{\infty}$-structure, and there are $A^{\infty}$-morphisms
$f: H(A) \to A$ and $g: A \to H(A)$.
We will describe versal NC deformation of $F$ using these $A^{\infty}$-algebras and morphisms.

In general, for $R \in (\text{Art}_k)$, we define $m_{R,n}: R \otimes_k H^1(A)^{\otimes n} \to R \otimes_k H^2(A)$,  
$f_{R,n}: R \otimes_k H^1(A)^{\otimes n} \to R \otimes_k H^1(A)$ and so on by the extensions of scalars.

We consider the {\em Maurer-Cartan equation} in $A^{\infty}$-algebras using the following proposition:

\begin{Prop}
Let $(R,M) \in (Art)$, and let $f: A \to B$ be an $A^{\infty}$-morphism.
Let $x \in M \otimes A^1$ and $y = \sum_{i \ge 1}f_{R,i}(x^{\otimes i}) \in M \otimes B^1$.
If $x$ satisfies the Maurer-Cartan equation
\[
\sum_{i \ge 1} m_{R,i}^A(x^{\otimes i}) = 0 \in R \otimes A^2
\]
then so does $y$: 
\[
\sum_{i \ge 1} m_{R,i}^B(y^{\otimes i}) = 0 \in R \otimes B^2.
\]
\end{Prop}

We note that the sums are finite because $M$ is nilpotent.

\begin{proof}
\[
\begin{split}
&\sum_{n \ge 1} m_n^B(y^{\otimes n}) \\
&= \sum_{n, i_1,\dots,i_n \ge 1} m_n^B(f_{i_1}(x^{\otimes i_1}) \otimes \dots \otimes f_{i_n}(x^{\otimes i_n})) \\
&= \sum_{n, i_1,\dots,i_n \ge 1} (-1)^{\sum_{j < k} i_j(1-i_k)} 
m_n^B(f_{i_1} \otimes \dots \otimes f_{i_n})(x^{\otimes (\sum_{j=1}^n i_j)}) \\
&= \sum_{r,t \ge 0, s \ge 1} (-1)^{rs+t}f_{r+1+t}(1^{\otimes r} \otimes m_s^A \otimes 1^{\otimes t})
(x^{\otimes (r+s+t)}) \\
&= \sum_{r,t \ge 0, s \ge 1} (-1)^{t}f_{r+1+t}(x^{\otimes r} \otimes m_s^A(x^{\otimes s}) \otimes x^{\otimes t}) \\
&= 0
\end{split}
\]
where we dropped the subscripts $R$ for simplicity. 
\end{proof}

In the above argument, we followed the {\em Koszul rule} of the signs:
\[
(x \otimes y)(z \otimes w) = (-1)^{\text{deg}(y)\text{deg}(z)} xz \otimes yw.
\]

\begin{Lem}
Let $A$ be a DG algebra, and let $f: H(A) \to A$ and $g: A \to H(A)$ be $A^{\infty}$-morphisms
such that $g \circ f = 1_{H(A)}$.
Let $(R,M) \in (Art_k)$, 
let $x \in M \otimes H^1(A)$, and let 
$y = \sum_{i \ge 1} f_{R,i}(x^{\otimes i}) \in M \otimes A^1$.
Then $x = \sum_{i \ge 1} g_{R,i}(y^{\otimes i})$.
\end{Lem}

\begin{proof}
We have 
\[
\begin{split}
&\sum_{n \ge 1}g_n(y^{\otimes n}) \\
&= \sum_{n, i_1,\dots,i_n \ge 1} g_n(f_{i_1}(x^{\otimes i_1}) \otimes \dots \otimes f_{i_n}(x^{\otimes i_n})) \\
&= \sum_{n, i_1,\dots,i_n \ge 1} (-1)^{\sum_{j < k} i_j(1-i_k)} 
g_n(f_{i_1} \otimes \dots \otimes f_{i_n})(x^{\otimes (\sum_{j=1}^n i_j)}) \\
&= \sum_{n \ge 1} (g \circ f)_n(x^{\otimes n}) \\
&= x.
\end{split}
\]
\end{proof}

\begin{Cor}
$x \in M \otimes H^1(A)$ satisfies the MC equation if and only if 
$y = \sum_{i \ge 1}f_{R,i}(x^{\otimes i}) \in M \otimes A^1$ satisfies the MC equation.
\end{Cor}

Now we construct a {\em versal deformation} over its parameter algebra, called the {\em deformation ring}.
Let $\{v_i\}$ be a basis of $H^1(A)$.
Let $x \in H^1(A)^* \otimes H^1(A)$ be the tautological element 
corresponding to the identity of $H^1(A)$.
Then we can write 
\[
x = \sum_i v_i^* \otimes v_i
\]
for the dual basis $\{v_i^*\}$ of $H^1(A)^*$. 

Let $m_n: T^nH^1(A) = (H^1(A))^{\otimes n} \to H^2(A)$ be the $A^{\infty}$-multiplication, 
let $m_{[2,n]} = \sum_{i=2}^n m_i: \bigoplus_{i=0}^n T^iH^1(A) \to H^2(A)$, and let 
$m_{[2,n]}^*: H^2(A)^* \to \bigoplus_{i=0}^n T^iH^1(A)^*$ be the dual map.
We define 
\[
R_n = \bigoplus_{i=0}^n T^iH^1(A)^*/m_{[2,n]}^*H^2(A)^*
\]
where $R_0 = k$, $R_1 = k \oplus H^1(A)^*$, and the product of total degree more than $n$ is set to be zero.
There are natural surjective ring homomorphisms $R_n \to R_{n'}$ for $n > n'$.
Let $M_n = Ker(R_n \to R_0 = k)$.
Then we have $M_n^{n+1} = 0$.
We define the formal completion
\[
\hat R = \varprojlim R_n = \hat T^{\bullet}H^1(A)^*/m^*H^2(A)^*
\]
where we set formally $m = \sum_{i=2}^{\infty} m_i$.

Let $m_{R_n,i}: R_n \otimes T^iH^1(A) \to R_n \otimes H^2(A)$ be the map obtained from $m_n$ by 
scalar extension.

\begin{Lem}
$R_n$ is the largest quotient ring of $\bigoplus_{i=0}^n T^iH^1(A)^*$ such that 
$\sum_{i=2}^n m_{R_n,i}(x^{\otimes i}) = 0$ in $R_n \otimes H^2(A)$.
\end{Lem}

\begin{proof}
Let $\{w_j\}$ be a basis of $H^2(A)$, and let $\{w_i^*\}$ be the dual basis of $H^2(A)^*$.
We write $m_k(v_{i_1} \otimes \dots \otimes v_{i_k}) = \sum_j a_{i_1,\dots,i_k,j}w_j$.
Then 
\[
\sum_{k=2}^n m_k^*(w_j^*) = \sum_{k, i_1,\dots,i_k} a_{i_1,\dots,i_k,j} v_{i_1}^* \otimes \dots \otimes v_{i_k}^*.
\]
We have 
\[
\begin{split}
&\sum_{k=2}^n m_{R,k}(x^{\otimes k}) = \sum_{k, i_1,\dots,i_k} m_{R,k}(v_{i_1}^* \otimes \dots \otimes v_{i_k}^* 
\otimes v_{i_1} \otimes \dots \otimes v_{i_k}) \\
&= \sum_{k, i_1,\dots,i_k, j} a_{i_1,\dots,i_k,j} v_{i_1}^* \otimes \dots \otimes v_{i_k}^* 
\otimes w_j.
\end{split}
\]
Therefore $\sum_{k=2}^n m_{R,k}(x^{\otimes k}) = 0$ in $R \otimes H^2(A)$ if and only if 
$\sum_{k=2}^n m_k^*(w_j^*) = 0$ in $R$ for all $j$.
\end{proof}

\begin{Cor}
Let $y_n = \sum_{i=1}^n f_{R_n,i}(x^{\otimes i}) \in R_n \otimes A^1$.
Then $(d_{R_n,I} + y_n)^2 = 0$ as an endomorphism of $R_n \otimes I^*$, 
where we denote $d_{R_n,I} = 1_{R_n} \otimes d_I$.
\end{Cor}

By Lemma \ref{MC}, we define an NC deformation
\[
F_n = \text{Ker}(d_{R_n,I}+y_n: R_n \otimes I^0 \to R_n \otimes I^1)
\]
of $F$ over $R_n$.

The following theorem is apparently well-known to experts (cf. \cite{Toda}):

\begin{Thm}\label{versal}
Let $\hat F = \varprojlim F_n$ be the inverse limit.
Then the formal deformation $\hat F$ over $\hat R$
is a versal non-commutative deformation of $F$
\end{Thm}

\begin{proof}
We have to prove the following statement: \lq\lq Let $R$ be a quotient algebra of $\bigoplus_{i=0}^n T^iH^1(A)^*$
such that $R_n$ is its quotient algebra.
Assume that there is an element $y \in R \otimes A^1$ which satisfies the Maurer-Cartan equation and induces 
$y_n$ on $R_n \otimes A^1$.
Then $R = R_n$''.

We will derive a contradiction assuming that $R \ne R_n$.
We may assume that 
the images of $R$ and $R_n$ to quotient algebras of $\bigoplus_{i=0}^{n-1} T^iH^1(A)^*$
coincide.
Let $y_R = \sum_{i=1}^n f_{R,i}(x^{\otimes i}) \in R \otimes A^1$ for 
$x = \sum v_i^* \otimes v_i \in H^1(A)^* \otimes H^1(A)$.
We set $z = y - y_R \in M^n \otimes A^1$.
We note that $y$ satisfies the MC equation over $R$ but $y_R$ does not, because neither does $x$ over $R$.

We have $y^{\otimes i} = (y_R + z)^{\otimes i} = y_R^{\otimes i}$ for $i \ge 2$, and
we have $x = \sum_{i=1}^n g_{R,i}(y_R^{\otimes i})$.
Let $x' = \sum_{i=1}^n g_{R,i}(y^{\otimes i})$.
Then we have $x' = x + g_{R,1}(z)$.
It follows that $x^{\otimes i} = x^{\prime \otimes i}$ for $i \ge 2$.
Since $y$ satisfies the MC equation over $R$, so does $x'$.
But since $m_1^{H(A)} = 0$, we deduce that $x$ satisfies the MC equation over $R$, 
a contradiction.
\end{proof}

The parameter algebra $\hat R$ of the versal deformation is called a {\em deformation algebra} of $F$.

%%%%%%%%%%%%%%%%%%%%%%%%%%%%%%%%%%%
%%%%%%%%%%%%%%%%%%%%%%%%%%%%%%%%%%%
%%%%%%%%%%%%%%%%%%%%%%%%%%%%%%%%%%%%%
\section{$1$-pointed versus $r$-pointed deformations}

Now we consider $r$-pointed deformations for a positive integer $r \ge 1$.
If $r = 1$, then they are NC deformations in the previous sections. 
It is a refined version in the case where the coherent sheaf $F$ has a direct sum decomposition to
$r$ ordered factors $F = \bigoplus_{i=1}^r F_i$.

We consider the base ring $k^r$, the product ring of $r$ copies of $k$, instead of $k$.
$F$ has a structure of a left $k^r$-module, 
where the orthogonal idempotents $e_i$ ($1 \le i \le r$) of $k^r$ correspond to the projections
$F \to F_i = e_iF$.

Let $(Art^r_k)$ be the category of pairs $(R,M)$ such that $R$ is an associative $k^r$-algebra 
with an augmentation $R \to k^r$ and $M$ is a
two-sided ideal satisfying the conditions that $R$ is a finite dimensional $k$-module, 
$R/M \cong k^r$, and that $M^{n+1} = 0$ for some $n$.
We have $R/M \cong \bigoplus_{i=1}^r R/M_i$ for maximal two-sided ideals $M_i$.
It follows that the $R/M_i$ are the only simple $R$-modules and any finitely generated $R$-module
is obtained as a successive extension of the $R/M_i$ (cf. \cite{NCdef}).

\begin{Defn}
Let $F = \bigoplus_{i=1}^r F_i$ be a direct sum of coherent sheaves with proper supports on an algebraic variety
$X$ and $R \in (Art^r_k)$.
Let $F_R$ be a left $R \otimes_k \mathcal{O}_X$-module
which is coherent as an $\mathcal{O}_X$-module.
Then a pair $(F_R, \phi)$ is said to be 
an {\em $r$-pointed non-commutative deformation} of $F$ over $R$, 
if $F_R$ is flat as a left $R$-module and $\phi: k^r \otimes_R F_R \to F$ is an isomorphism.
\end{Defn}

The injective resolution $F \to I^{\bullet}$ are $k^r$-equivariant in the sense that 
$I^{\bullet} = \bigoplus_{i=1}^r I_i^{\bullet}$, where the $F_i \to I_i^{\bullet}$ are injective resolutions.
The graded ring $A = \text{Hom}^{\bullet}(I^{\bullet},I^{\bullet})$ 
has a structure of $k^r$-bimodules; we have a direct sum decomposition
$\text{Hom}^{\bullet}(I^{\bullet},I^{\bullet}) = \bigoplus_{i,j=1}^r \text{Hom}^{\bullet}(I_i^{\bullet},I_j^{\bullet})$.
It is convenient to write $A$ in a matrix form $A_{ij} = \text{Hom}^{\bullet}(I_i^{\bullet},I_j^{\bullet})$, because 
we have $A_{ij}A_{kl} = 0$ if $j \ne k$.

The constructions of the deformation ring and the versal deformation are generalized from the $1$-pointed case
to the $r$-pointed case in the following way.
The cohomology groups $H^p(A) = \text{Ext}^p(F,F)$ have also $k^r$-bimodule structures.
$H(A) = \bigoplus_i H^i(A)$ has an $A^{\infty}$-structure with a $k^r$-bimodule structure.
If $n > 0$, then there is an injective homomorphism from a direct summand
\[
T^n_{k^r}H^1(A) \to T^n_kH^1(A)
\]
between tensor products.
For example, 
\[
T^2_{k^r}H^1(A) = \bigoplus_{i,j,k} H^1(A)_{ij} \otimes H^1(A)_{jk} \subset T^2_kH^1(A).
\]
The $A^{\infty}$-multiplications
\[
m^r_n: T^n_{k^r}H^1(A) \to H^2(A)
\]
for $n \ge 2$ are induced from the $1$-pointed case $m_n = m_n^1$.
We define
\[
R^r_n = \bigoplus_{i=0}^n T^i_{k^r}H^1(A)^*/(m^r_{[2,n]})^*H^2(A)^*
\]
for $m^r_{[2,n]} = \sum_{i=2}^n m^r_i$, and
\[
\hat R^r = \varprojlim R^r_n
\]
as before.

In order to define $\hat F^r$, we take the tautological element 
\[
x = \sum_i v_i^* \otimes v_i \in H^1(A)^* \otimes_{k^r} H^1(A)
\]
again, where each $v_i$ belongs to a $H^1(A)_{st}$ for some $s,t$ so that $v_i^*$ belongs to $(H^1(A)^*)_{ts}$.
Let 
\[
y^r_n = \sum_{i=1}^n f^r_{R^r_n,i}(x^{\otimes i}) \in R^r_n \otimes A^1
\] 
where $f^r_{R^r_n,i}$ is induced from $f_i$.
Then we define
\[
F^r_n = \text{Ker}(d_{R^r_n,I}+y_n: R^r_n \otimes_{k^r} I^0 \to R^r_n \otimes_{k^r} I^1)
\]
and 
\[
\hat F^r = \varprojlim F^r_n.
\]

We compare deformation rings 
\[
\hat R^1 = \varprojlim R^1_n, \quad \hat R^r = \varprojlim R^r_n
\]
of $1$-pointed and $r$-pointed deformations.
Their Zariski cotangent spaces are the same $H^1(A)^* = (\text{Ext}^1(F,F))^*$.
The truncated deformation ring $R^r_n$ of $r$-pointed deformations is a quotient algebra of the 
tensor algebra over $k^r$:
\[
\begin{split}
&T^{\bullet}_{k^r}H^1(A)^* = \prod_{i = 0}^{\infty} T^i_{k^r}H^1(A)^* \\
&= k^r \times H^1(A)^* \times (H^1(A)^* \otimes_{k^r} H^1(A)^*) \times \dots
\end{split}
\]
where the tensor products are taken over the base ring $k^r$.

There is a split surjective ring homomorphism $k^r \times T^{\bullet}_kH^1(A)^* \to T^{\bullet}_{k^r}H^1(A)^*$. 
We have 
\[
T^{\bullet}_{k^r}H^1(A)^* = (k^r \times T^{\bullet}_kH^1(A)^*)/J
\]
where the ideal $J$ is generated by relations $\sum_{i=1}^r e_i = 1$ and $H^1(A)^*_{ij}H^1(A)^*_{kl}=0$ 
for $j \ne k$.
The degree $0$ part of $T^{\bullet}_{k^r}H^1(A)^*$ is $k^r$, which is larger than $k$, but positive degree parts are
quotients of the usual tensor products $T^i_kH^1(A)^*$.
Therefore the $r$-pointed deformation ring $\hat R^r$ is not exactly a quotient of the 
$1$-pointed deformation ring $\hat R^1$, but almost is.
In particular, $r$-pointed deformations are derived from a special case of $1$-pointed deformations.

We note that the deformation $F_{R^r_n}$ over $R^r_n$ is different from the one 
induced from the deformation $F_{R^1_n}$ by the natural ring homomorphism
$R^1_n \to R^r_n$.
For example, $F_{R^r_n}$ is flat over $R^r_n$ and $k^r \otimes_{R^r_n} F_{R^r_n} = F$, but 
$k \otimes_{R^1_n} F_{R^1_n} = F$.
We have 
\[
F_{R^r_n} = (R^r_n \otimes_{R^1_n} F_{R^1_n})/N
\]
where $N$ is a submodule consisting of irrelevant factors of $F_{R^1_n}$ 
that are attached in the extension process; we have to attach $F_i$'s instead of $F$ (cf. Example~\ref{two lines}).

The following theorem is a consequence of Theorem \ref{versal}:

\begin{Thm}\label{r-pointed}
The formal deformation $\hat F^r$ of $F$ over $\hat R^r$
is a versal $r$-pointed non-commutative deformation of $F$ in the following sense.
If $(F^r_R,\phi^r_0)$ is any $r$-pointed non-commutative deformation over $(R,M) \in (\text{Art}_k^r)$ 
such that $\phi^r_0: R/M \otimes_{k^r} F_R \to F$ is an isomorphism,
then there exist an integer $n$ and a $k^r$-algebra homomorphism $\psi^r: R^r_n \to R$ such that
there is an isomorphism $\phi^r: R \otimes_{R^r_n} F^r_n \to F_R$ which induces $\phi^r_0$ over $R/M$.
\end{Thm}

%%%%%%%%%%%%%%%%%%%%%%%%%%%%%%%%%%
%%%%%%%%%%%%%%%%%%%%%%%%%%%%%%%%%%
%%%%%%%%%%%%%%%%%%%%%%%%%%%%%%%%%%
\section{Remark on universal extensions}

We consider iterated self extensions of $F = \bigoplus_{i=1}^r F_i$ in this section.
NC deformations of $F$ are iterated self extensions of $F$.
Conversely, any iterated self extensions of $F$ are expected to be expressed as 
NC deformations of $F$, and the versal deformation is given by a tower of universal extensions.
Indeed if $F$ is a {\em simple collection}, i.e., if $\text{End}(F) \cong k^r$, then it is the case 
(\cite{NCdef}~Theorem~4.8).
The point is that the parameter algebra in this case is naturally given as 
the endomorphism ring of the iterated non-trivial self extensions.

We define inductively a tower of {\em universal extensions} by $E^r_0 = F$ and 
\begin{equation}\label{ext0}
0 \to \text{Ext}^1(E^r_n,F)^* \otimes_{k^r} F \to E^r_{n+1} \to E^r_n \to 0
\end{equation}
for $n \ge 0$, or equivalently 
\[
0 \to \bigoplus_j \text{Ext}^1(E^r_{n,i},F_j)^* \otimes_k F_j \to E^r_{n+1,i} \to E^r_{n,i} \to 0
\]
where we note that $E^r_n = \bigoplus_{i=1}^r E^r_{n,i}$ is a left $k^r$-module 
and $\text{Ext}^1(E^r_n,F)$ is a $k^r$-bimodule.
The above exact sequence corresponds to a natural morphism 
\[
\begin{split}
&\text{Ext}^1(E^r_n,F)[-1] \otimes_{k^r} E^r_n = \bigoplus_{i,j} \text{Ext}^1(E^r_{n,i},F_j)[-1] \otimes_k E^r_{n,i} \\
&\to F = \bigoplus_j F_j
\end{split}
\]
in the derived category.

On the other hand, in the notation of the previous sections, from an exact sequence
\[
0 \to (M^r_{n+1})^{n+1} \to R^r_{n+1} \to R^r_n \to 0
\]
we obtain an exact sequence
\begin{equation}\label{ext2}
0 \to (M^r_{n+1})^{n+1} \otimes_{k^r} F \to F^r_{n+1} \to F^r_n \to 0
\end{equation}
where we note that $(M^r_{n+1})^{n+1} \otimes_{R^r_{n+1}} F \cong (M^r_{n+1})^{n+1} \otimes_{k^r} F$.

We expect that (\ref{ext0}) and (\ref{ext2}) are isomorphic as exact sequences of $\mathcal{O}_X$-modules.
For example, we have $M_1 \cong M^r_1 \cong \text{Ext}^1(F,F)^*$, and this is the case for $n = 0$.

In the case $n = 1$, from an exact sequence
\[
\begin{split}
&M_1^* \otimes \text{Hom}(F,F) \to \text{Ext}^1(F,F) \to \text{Ext}^1(F_1,F) \\
&\to M_1^* \otimes \text{Ext}^1(F,F) \to \text{Ext}^2(F,F)
\end{split}
\]
we deduce that
\[
\text{Ext}^1(F_1,F) = \text{Ker}(M_1^* \otimes \text{Ext}^1(F,F) \to \text{Ext}^2(F,F)).
\]
Thus 
\[
\text{Ext}^1(F_1,F)^* = \text{Coker}(m_2: \text{Ext}^2(F,F)^* \to (\text{Ext}^1(F,F)^*)^{\otimes 2}) = M_2^2.
\]
Thus each corresponding terms in (\ref{ext0}) and (\ref{ext2}) coincide for $n = 1$.

\vskip 1pc

We compare universal extensions corresponding to $1$-pointed and $r$-pointed deformations:

\begin{Lem}
There are natural split surjective homomorphisms $E^1_n \to E^r_n$.
\end{Lem}

\begin{proof}
For $n = 0$, we have $E^1_0 = E^r_0 = F$.

Assume that there is a split surjective homomorphism $E^1_n \to E^r_n$.
Then there is an induced split surjective homomorphism
\[
\text{Ext}^1(E^1_n,F)^* \to \text{Ext}^1(E^r_n,F)^*.
\]
Since 
\[
\begin{split}
&\text{Ext}^1(E^r_n,F)^* \otimes_k F = \sum_{i,j} \text{Ext}^1(E^r_n,F_i)^* \otimes_k F_j \\
&\text{Ext}^1(E^r_n,F)^* \otimes_{k^r} F = \sum_i \text{Ext}^1(E^r_n,F_i)^* \otimes_k F_i
\end{split}
\]
we have a split surjective homomorphism 
\[
\text{Ext}^1(E^r_n,F)^* \otimes_k F \to \text{Ext}^1(E^r_n,F)^* \otimes_{k^r} F
\] 
hence there is also a split surjective homomorphism $E^1_{n+1} \to E^r_{n+1}$.
\end{proof}

%%%%%%%%%%%%%%%%%%%%%%%%%%%%%%%%%%%
%%%%%%%%%%%%%%%%%%%%%%%%%%%%%%%%%%%
%%%%%%%%%%%%%%%%%%%%%%%%%%%%%%%%%%%%
\section{Examples}

We consider some examples of versal NC deformations in this section.
We start with a trivial example:

\begin{Expl}
Let $F = \mathcal{O}_x$ be the structure sheaf of a point $x \in X$.
We claim that the versal deformation $\hat F$ of $F$ is isomorphic to the deformation algebra $\hat R$, which is 
commutative and isomorphic to the formal completion of the local ring $\hat O_{X,x}$. 

$F$ is a simple collection with $r = 1$, i.e., a simple sheaf in this case, 
hence the versal deformation is given by the tower of universal extensions
(\cite{NCdef}~Theorem~4.8).
Therefore it is sufficient to prove that any NC deformation $F_R$ of $F$ over some $R \in (\text{Art}_k)$ obtained by 
successive non-trivial extensions is of the form
$F_R \cong \mathcal{O}_X/J$ for an ideal $J$ such that $\text{Supp}(F_R) = \{x\}$.
We proceed by induction on $\dim R$.
Let 
\[
0 \to \mathcal{O}_x \to E \to \mathcal{O}_X/J \to 0
\]
be a non-trivial extension.
Since $Ext^1(\mathcal{O}_X,\mathcal{O}_x) = 0$, the natural surjective homomorphism $\mathcal{O}_X \to \mathcal{O}_X/J$ lifts to 
a homomorphism $\mathcal{O}_X \to E$.
Let $J'$ be the kernel.
Then there are homomorphism $\mathcal{O}_X/J' \to E \to \mathcal{O}_X/J$ whose combination is surjective 
and the first homomorphism is injective.
There are two cases: $\text{length}(\mathcal{O}_X/J') - \text{length}(\mathcal{O}_X/J) = 0$ or $1$.
In the first case, we have $J = J'$ and the homomorphism $E \to \mathcal{O}_X/J$ splits, 
a contradiction to the hypothesis that the extension is non-trivial.
In the second case, we have $E = \mathcal{O}_X/J'$, and the claim is proved after taking the inverse limit.  
\end{Expl} 

\begin{Rem}
Let $F = \mathcal{O}_x$ for a smooth point $x \in X$.
Then we have $Ext^p(\mathcal{O}_x,\mathcal{O}_x) \cong \wedge^p k^n$ for $n = \dim X$.
The deformation algebra is isomorphic to the formal power series ring $k[[x_1,\dots,x_n]]$.
There are no obstructions for commutative deformations of $F$, but the non-commutative deformations 
are highly obstructed.
\end{Rem}

\begin{Expl}\label{two lines}
Let $X = \{xy = 0\} \subset \mathbf{P}^2$ be the union of two distinct lines, and let  
$F = \mathcal{O}_X/(x) \oplus \mathcal{O}_X/(y) := F_x \oplus F_y$ be the sum of structure sheaves of these lines.
We compare $1$-pointed and $2$-pointed deformations of $F$.

The $2$-pointed deformation ring is calculated in \cite{NCdef}~Example~5.5:
\[
\hat R^2 = \left( \begin{matrix} k & kt \\ kt & k \end{matrix} \right) \mod (t^2)
\]
and $\dim \hat R^2 = 4$.
On the other hand, the $1$-pointed deformation ring is:
\[
\hat R^1 = k[[x,y]]/(xy) = k\langle \langle x,y \rangle \rangle/(xy,yx)
\]
and $\dim \hat R^1 = \infty$.

The corresponding deformations are as follows.
There are non-trivial extensions
\[
\begin{split}
&0 \to \mathcal{O}_X/(y) \to F^2_{1,x} \to \mathcal{O}_X/(x) \to 0 \\
&0 \to \mathcal{O}_X/(x) \to F^2_{1,y} \to \mathcal{O}_X/(y) \to 0
\end{split}
\]
where $F^2_{1,x}$ (resp. $F^2_{1,y}$) are invertible sheaves on $X$ whose degrees on the irreducible components 
$L_x = \{y=0\}$ and $L_y = \{x = 0\}$
of $X$ are $(0,1)$ (resp. $(1,0)$).
We have
\[
\hat F^2 = F^2_1 = F^2_{1,x} \oplus F^2_{1,y}.
\]

On the other hand, we have
\[
F^1_n \cong (F^2_{1,x})^{\oplus n} \oplus \mathcal{O}_X/(x) \oplus (F^2_{1,y})^{\oplus n} \oplus \mathcal{O}_X/(y).
\]
\end{Expl}

\begin{Expl} 
Let $X = \{x^2+y^2+y^3=0\} \subset \mathbf{P}^2$ be a rational curve with one node, and let 
$F$ be the structure sheaf of the normalization of $X$.

$X$ has a singularity which is analytically isomorphic to the singularity of the 
variety considered in the previous example.
We have again $\hat R = k[[x,y]]/(xy)$, and $\hat F$ becomes an invertible sheaf on an infinite chain of
rational curves. 

We have $\text{End}_X(F) = k$ but $\text{End}_{D_{sg}}(F) = k[t]/(t^2+1)$ (\cite{ODP}).
\end{Expl}

\begin{Expl}[\cite{NCdef}~Example~5.8]
Let $X = \mathbf{P}(1,2,3)$ be a weighted projective plane, and let $F = \mathcal{O}_X(1)$ be 
a reflexive sheaf of rank $1$ corresponding to a line on $X$ connecting its two singular points.

Then the deformation ring for NC deformations is an infinite dimensional algebra
$\hat R = k \langle \langle x,y \rangle \rangle/(x^2,y^3)$, 
where the generators $x,y$ respectively correspond to local extensions of $F$ near the singularities of types
$\frac 12(1,1)$ and $\frac 13(1,2)$.
But $\hat R^{ab} = k[[x]]/(t^6)$ is finite dimensional for commutative deformations.
\end{Expl}

\begin{Lem}\label{relation ideal}
Let $X$ be a proper variety and let $Y$ be a closed subvariety.
Let $F = \mathcal{O}_Y = \mathcal{O}_X/J$ be the structure sheaf of $Y$ regarded as an $\mathcal{O}_X$-module.
Assume that $H^0(X,\mathcal{O}_X) \cong H^0(X,\mathcal{O}_Y) \cong k$ and $H^1(X,\mathcal{O}_Y) = 0$.
Then the versal deformation $\hat F$ of $F$ is of the form 
\[
\hat F = \varprojlim(R_n \otimes \mathcal{O}_X)/J_n
\]
for left $R_n \otimes \mathcal{O}_X$-ideals $J_n$.
Moreover if $J \otimes \mathcal{O}_X(1)$ is generated by global sections $h_i$ ($i = 1,\dots,r$) 
for a very ample invertible sheaf $\mathcal{O}_X(1)$ on $X$ 
and if $H^1(X,J \otimes \mathcal{O}_X(1)) = 0$, then there are global sections
$\hat h_i$ of $\varprojlim (J_n \otimes \mathcal{O}_X(1))$ which generate $\varprojlim (J_n \otimes \mathcal{O}_X(1))$
and induce $h_i \in J \otimes \mathcal{O}_X(1)$. 
\end{Lem}

\begin{proof}
Since $F$ is a simple sheaf, a versal NC deformation is obtained by a sequence of universal extensions. 
We prove that a deformation $F_R$ over $R \in (\text{Art}_k)$ is of the form $(R \otimes \mathcal{O}_X)/J_R$ 
for a left $R \otimes \mathcal{O}_X$-ideal $J_R$ by induction on $\dim_k R = i$.
Let 
\[
0 \to \mathcal{O}_Y \to E \to (R' \otimes \mathcal{O}_X)/J_{R'} \to 0
\]
be a non-trivial extension of NC deformations over an extension $0 \to R/M \to R \to R' \to 0$ 
of parameter algebras.
Since $\text{Ext}^1(\mathcal{O}_X,\mathcal{O}_Y) = 0$, the natural homomorphism 
$R \otimes \mathcal{O}_X \to (R' \otimes \mathcal{O}_X)/J_{R'}$ lifts to 
an $\mathcal{O}_X$-homomorphism $R \otimes \mathcal{O}_X \to E$.
Thus there is a commutative diagram
\[
\begin{CD}
0 @>>> \mathcal{O}_X @>>> \mathcal{O}_X^{\oplus i} @>>> \mathcal{O}_X^{\oplus (i-1)} @>>> 0 \\
@. @VVV @VVV @VVV \\ 
0 @>>> \mathcal{O}_Y @>>> E @>>> (R' \otimes \mathcal{O}_X)/J_{R'} @>>> 0
\end{CD}
\]
of $\mathcal{O}_X$-modules.
Since the vertical arrows at both ends are surjective, so is the middle vertical arrow.

Since $H^0(X,\mathcal{O}_X) \cong H^0(X,\mathcal{O}_Y) \cong k$, the natural homomorphism
\[
H^0(\mathcal{O}_X^{\oplus i}) \to H^0(E) \cong R.
\]
is an isomorphism.
Using this isomorphism, we define a left $R$-module structure on $\mathcal{O}_X^{\oplus i}$.
Then the middle vertical arrow becomes a homomorphism of left $R \otimes \mathcal{O}_X$-modules, 
and we have $E \cong (R \otimes \mathcal{O}_X)/J_R$ for a left $R \otimes \mathcal{O}_X$-ideal
$J_R$.

We prove that the generating sections of $J \otimes \mathcal{O}_X(1)$ extend to generating sections of 
$J_R \otimes \mathcal{O}_X(1)$ by induction again.
From an exact sequence of kernels
\[
0 \to J \otimes \mathcal{O}_X(1) \to J_R \otimes \mathcal{O}_X(1) \to J_{R'} \otimes \mathcal{O}_X(1) \to 0
\]
we deduce that the homomorphism $H^0(X,J_R \otimes \mathcal{O}_X(1)) \to H^0(X,J_{R'} \otimes \mathcal{O}_X(1))$ 
is surjective, 
hence the global sections are liftable.
By Nakayama's lemma, they are generating.
\end{proof}

The following lemma says that the NC deformations of a Cartier divisor is not interesting:

\begin{Lem}\label{divisor}
Let $F = \mathcal{O}_D$ be the structure sheaf of a Cartier divisor $D \subset X$.
Assume that $H^2(\mathcal{O}_D)=H^1(\mathcal{O}_D(D)) = 0$.
Then the NC deformations of $F$ are unobstructed, i.e., the deformation ring is isomorphic to 
a non-commutative 
formal power series ring $k\langle \langle x_1,\dots,x_m \rangle \rangle$ for $m = \dim \text{Ext}^1(F,F)$. 
\end{Lem}

\begin{proof}
We have an exact sequence $0 \to \mathcal{O}_X(-D) \to \mathcal{O}_X \to \mathcal{O}_D \to 0$.
Then there is an exact sequence
\[
\text{Ext}^1(\mathcal{O}_X(-D),F) \to \text{Ext}^2(F,F) \to \text{Ext}^2(\mathcal{O}_X,F).
\]
Hence $\text{Ext}^2(F,F) \cong 0$.
\end{proof}

Therefore we consider NC deformations of higher codimensional subvarieties:

\begin{Expl}\label{lines}
Let $X = \mathbf{P}^n$ be a projective space with homogeneous coordinates $[x_1,\dots,x_{n+1}]$, and
let $F = \mathcal{O}_L = k[x_1,\dots,x_{n+1}]/(x_1,\dots,x_{n-1})\tilde {}$ be the structure sheaf of a line $L$, where
$\tilde {}$ denotes a coherent sheaf on $X$ associated to a graded module.

We claim that the deformation algebra is given by
\[
\hat R = k\langle \langle a_1,b_1,\dots,a_{n-1},b_{n-1} \rangle \rangle/
(a_ia_j-a_ja_i,b_ib_j-b_jb_i,a_ib_j-b_ja_i-a_jb_i+b_ia_j).
\]
and the versal deformation is given as a quotient by a left ideal:
\[
\hat F = \hat R[x_1,\dots,x_{n+1}]/(x_1+a_1x_n+b_1x_{n+1}, \dots, x_{n-1}+a_{n-1}x_n+b_{n-1}x_{n+1})\tilde {}
\]
where $\tilde {}$ denotes a coherent $\hat R \otimes \mathcal{O}_X$-module associated to a graded module.

We use Lemma~\ref{relation ideal}.
The sheaf $J \otimes \mathcal{O}_X(1)$ for the ideal sheaf $J$ of $L \subset X$ 
is generated by global sections $x_1,\dots,x_{n-1}$ and 
$H^1(X,J \otimes \mathcal{O}_X(1)) = 0$.
Hence $F_R$ should be of the form $(R \otimes \mathcal{O}_X)/J_R$ for an ideal sheaf $J_R$
such that $J_R \otimes \mathcal{O}_X(1)$ is generated by the following global sections which are linear forms on the $x_i$:
\[
x_1+a_1x_n+b_1x_{n+1}, \dots, x_{n-1}+a_{n-1}x_n+b_{n-1}x_{n+1}
\]
where we note that elements of the form $1 + r$ with $r \in M$ are invertible, so that the coefficients can be reduced 
to the above form.

Let $\hat R^{\text{ab}}$ be the maximal abelian quotient of $\hat R$.
Then it is the completed local ring of 
a Grassmann variety $G(2,n+1)$ at a point.
Since $\hat R$ and $\hat R^{\text{ab}}$ have the same Zariski cotangent spaces,  
the variables of $\hat R$ are the $a_i,b_i$ as in the above expression of $\hat F$.
Since $\hat R^{\text{ab}}$ is a smooth commutative ring, the relations for $\hat R$ are contained in the 
commutator ideal of the variables.

In order to determine the quadratic terms in the relations, we calculate 
\[
m_2: \text{Ext}^1(F,F) \times \text{Ext}^1(F,F) \to \text{Ext}^2(F,F)
\]
explicitly.
We have 
\[
\begin{split}
&\text{Ext}^1(F,F) \cong k^{2(n-1)} \cong H^0(N_{L/X}) \\ 
&\text{Ext}^2(F,F) \cong k^{3(n-1)(n-2)/2} \ne H^1(N_{L/X}) = 0
\end{split}
\]
where $N_{L/X}$ is the normal bundle of $L$ in $X$.
Let $t_1,t_2$ be the homogeneous coordinates on $L$, and let $\nu_1,\dots,\nu_{n-1}$ be the normal directions
of $L$.
Then $\text{Ext}^1(F,F)$ has a basis $v_{ij} = t_1^it_2^{1-i}\nu_j$ ($i=0,1$, $1 \le j \le n-1$), and 
$\text{Ext}^2(F,F)$ has a basis $w_{ijk} = t_1^it_2^{2-i}\nu_j \wedge \nu_k$ ($i = 0,1,2$, $1 \le j < k \le n-1$).
Therefore $m_2$ is surjective and its kernel has a basis
\[
\begin{split}
&v_{0j}v_{0j}, \,\, v_{0j}v_{1j}, \,\, v_{1j}v_{0j}, \,\, v_{1j}v_{1j}, \,\, v_{0j}v_{0k} + v_{0k}v_{0j}, \\
&v_{0j}v_{1k} + v_{1k}v_{0j}, \,\, v_{1j}v_{0k} + v_{0k}v_{1j}, \,\, v_{1j}v_{1k} + v_{1k}v_{1j}, \,\, 
v_{0j}v_{1k} - v_{1j}v_{0k}
\end{split}
\]
where $1 \le j \le n-1$ for the first $4$ terms, and $1 \le j < k \le n-1$ for the rest.
The dual basis of $\text{Im}(m_2^*) = \text{Ker}(m_2)^{\perp} \subset (\text{Ext}^1(F,F)^*)^{\otimes 2}$ is 
given by 
\[
a_ia_j-a_ja_i, b_ib_j-b_jb_i, a_ib_j-b_ja_i-a_jb_i+b_ia_j
\]
for $1 \le i < j \le n-1$, where $\{a_i,b_j\}_{i,j} \subset \text{Ext}^1(F,F)^*$ is the dual basis of $\{v_{0i},v_{1j}\}_{i,j}$.
They are the leading terms of the relations for $\hat R$.

Now we prove that there are no higher order terms in the relations, i.e., we prove that there is no
higher Massey products.
We use the fact that the variables $x_1,\dots, x_{n+1}$ in $\hat F$ are commutative.
We have in $\hat F$, 
\[
\begin{split}
&0 = x_ix_j-x_jx_i \\
&= (a_ia_j-a_ja_i)x_n^2 + (b_ib_j-b_jb_i)x_{n+1}^2 + (a_ib_j-b_ja_i-a_jb_i+b_ia_j)x_nx_{n+1}.
\end{split}
\]
Therefore we have
\[
a_ia_j-a_ja_i = b_ib_j-b_jb_i = a_ib_j-b_ja_i-a_jb_i+b_ia_j = 0
\]
in $\hat F$.
If there were higher order terms in the relations of $\hat R$ on top of the quadratic relations above, 
then there were more relations of order $\ge 3$, a contradiction to the fact that the relations
are given by $m^*\text{Ext}^2(F,F)$, and their number is $3(n-1)(n-2)/2$.  
Thus the claim is proved.

In particular, if $n \ge 3$, then the NC deformations of $F$ are obstructed, 
because there are non-trivial relations for $\hat R$,   
but there are more NC deformations than commutative deformations.

For example, if $n = 3$, then lines on $\mathbf{P}^3$ are parametrized by $G(2,4)$ under commutative deformations,
but the deformation ring for NC deformations is
\[
\hat R = k\langle \langle a,b,c,d \rangle \rangle/(ab-ba,cd-dc,ad-da-bc+cb).
\] 

We note that this kind of examples are not artificial.
For example, if we consider a Calabi-Yau manifold $Y$ such that $L \subset Y \subset \mathbf{P}^n$, then 
the deformation ring of $\mathcal{O}_L$ on $Y$, which is an important invariant of an analytic neighborhood of $L$ in $Y$, 
is a quotient ring of $\hat R$ (cf. \cite{DW}).
In this sense, it is interesting to calculate versal deformations of rational normal curves of higher degrees.
\end{Expl}

\newpage

\begin{Expl}\label{conics}
Let $X = \mathbf{P}^4$ with homogeneous coordinates $[x,y,z,w]$, and let 
$F = \mathcal{O}_L = k[x,y,z,w,t]/(x,y,zt-w^2)\tilde {}$ for a conic $L$ in $X$.

We claim that the deformation ring $\hat R$ of $F$ is given by

%\newpage

\[
\begin{split}
&\hat R = k\langle \langle  a_0,a_1,a_2,b_0,b_1,b_2,c_0,\dots,c_4 \rangle \rangle
/(a_0b_0-b_0a_0+(a_1b_1-b_1a_1)c_0, \\
&a_0b_1-b_1a_0+a_1b_0-b_0a_1+(a_1b_1-b_1a_1)c_1, \\
&a_0b_2-b_2a_0+a_1b_1-b_1a_1+a_2b_0-b_0a_2 +(a_1b_1-b_1a_1)c_2, \\
&a_1b_2-b_2a_1+a_2b_1-b_1a_2 + (a_1b_1-b_1a_1)c_3, \\
&a_2b_2-b_2a_2 + (a_1b_1-b_1a_1)c_4, \\
&a_0c_0-c_0a_0 + (a_1c_1-c_1a_1)c_0, \\
&a_0c_1-c_1a_0+a_1c_0-c_0a_1 +(a_1c_1-c_1a_1)c_1, \\
&a_0c_2-c_2a_0+a_1c_1-c_1a_1 + a_2c_0-c_0a_2 + (a_1c_1-c_1a_1)c_2 + (a_1c_3-c_3a_1)c_0, \\
&a_0c_3-c_3a_0+a_1c_2-c_2a_1+a_2c_1-c_1a_2 + (a_1c_1-c_1a_1)c_3 + (a_1c_3-c_3a_1)c_1, \\
&a_0c_4-c_4a_0+a_1c_3-c_3a_1 +a_2c_2-c_2a_2 +  (a_1c_1-c_1a_1)c_4 + (a_1c_3-c_3a_1)c_2, \\
&a_1c_4-c_4a_1+a_2c_3-c_3a_2 + (a_1c_3-c_3a_1)c_3, \\
&a_2c_4-c_4a_2 + (a_1c_3-c_3a_1)c_4, \\
&b_0c_0-c_0b_0 + (b_1c_1-c_1b_1)c_0, \\
&b_0c_1-c_1b_0+b_1c_0-c_0b_1 +(b_1c_1-c_1b_1)c_1, \\
&b_0c_2-c_2b_0+b_1c_1-c_1b_1 + b_2c_0-c_0b_2 + (b_1c_1-c_1b_1)c_2 + (b_1c_3-c_3b_1)c_0, \\
&b_0c_3-c_3b_0+b_1c_2-c_2b_1+b_2c_1-c_1b_2 + (b_1c_1-c_1b_1)c_3 + (b_1c_3-c_3b_1)c_1, \\
&b_0c_4-c_4b_0+b_1c_3-c_3b_1 +b_2c_2-c_2b_2 +  (b_1c_1-c_1b_1)c_4 + (b_1c_3-c_3b_1)c_2, \\
&b_1c_4-c_4b_1+b_2c_3-c_3b_2 + (b_1c_3-c_3b_1)c_3, \\
&b_2c_4-c_4b_2 + (b_1c_3-c_3b_1)c_4)
\end{split}
\]
and the versal deformation $\hat F$ is given by 
\[
\begin{split}
&\hat F = R[x,y,z,w,t]/(x+a_0z+a_1w+a_2t,\,y+b_0z+b_1w+b_2t, \\
&zt-w^2+c_0z^2+c_1zw+c_2zt+c_3wt+c_4t^2).
\end{split}
\]
We note that there are order $3$ terms in the relations of $\hat R$, i.e., $m_3 \ne 0$, but
$m_i = 0$ for $i \ge 4$. 

In order to prove the claim, we argue similarly to the previous example.
We have $N_{L/\mathbf{P}^4} \cong \mathcal{O}(2)^2 \oplus \mathcal{O}(4)$, and 
\[
0 \to \mathcal{O}_X(-4) \to \mathcal{O}_X(-3)^2 \oplus \mathcal{O}_X(-2) \to \mathcal{O}_X(-1)^2 \oplus \mathcal{O}_X(-2) 
\to \mathcal{O}_X \to \mathcal{O}_L \to 0.
\]
Hence
\[
\begin{split}
&\text{Ext}^1(F,F) = H^0(\mathbf{P}^1,\mathcal{O}(2)^2 \oplus \mathcal{O}(4)) \cong k^{11} \\
&\text{Ext}^2(F,F) = H^0(\mathbf{P}^1,\mathcal{O}(4) \oplus \mathcal{O}(6)^2) \cong k^{19}.
\end{split}
\]
$\hat F$ is written in the above form by Lemma~\ref{relation ideal}.
We will determine the relations among variables $a_i,b_j,c_k$ in $\hat R$.

The quadratic terms of the relations are determined by the multiplication
\[
m_2: \text{Ext}^1(\mathcal{O}_L,\mathcal{O}_L) \otimes \text{Ext}^1(\mathcal{O}_L,\mathcal{O}_L) 
\to \text{Ext}^2(\mathcal{O}_L,\mathcal{O}_L).
\]
We take the dual basis 
\[
a^*_0,a^*_1,a^*_2,b^*_0,b^*_1,b^*_2,c^*_0, c^*_1,c^*_2,c^*_3,c^*_4
\]
of $\text{Ext}^1(\mathcal{O}_L,\mathcal{O}_L)$, and a basis
\[
d^*_0,d^*_1,d^*_2,d^*_3,d^*_4,e^*_0,e^*_1,e^*_2,e^*_3,e^*_4,e^*_5,e^*_6,f^*_0,f^*_1,f^*_2,f^*_3,f^*_4,f^*_5,f^*_6
\]
of $\text{Ext}^2(\mathcal{O}_L,\mathcal{O}_L)$, 
so that the multiplication map satisfies the following:
\[
\begin{split}
&m_2(a^*_i,a^*_j) = m_2(b^*_i,b^*_j) = m_2(c^*_i,c^*_j) = 0 \\
&m_2(a^*_i,b^*_j) = - m_2(b^*_j,a^*_i) = d^*_{i+j} \\
&m_2(a^*_i,c^*_j) = - m_2(c^*_j,a^*_i) = e^*_{i+j} \\
&m_2(b^*_i,c^*_j) = - m_2(c^*_j,b^*_i) = f^*_{i+j}.
\end{split}
\]
Therefore the image of the map $m_2^*:\text{Ext}^2(\mathcal{O}_L,\mathcal{O}_L)^* \to (\text{Ext}^1(\mathcal{O}_L,\mathcal{O}_L)^*)^{\otimes 2}$ 
is spanned by 
the following:
\[
\begin{split}
&a_0b_0-b_0a_0, \, a_0b_1-b_1a_0 + a_1b_0-b_0a_1, \, a_0b_2-b_2a_0 + a_1b_1-b_1a_1 + a_2b_0-b_0a_2, \\
&a_1b_2-b_2a_1 + a_2b_1-b_1a_2, \, a_2b_2-b_2a_2, \\
&a_0c_0-c_0a_0, \, a_0c_1-c_1a_0 + a_1c_0-c_0a_1, \, a_0c_2-c_2a_0 + a_1c_1-c_1a_1 + a_2c_0-b_0c_2, \\
&a_0c_3-c_3a_0 + a_1c_2-c_2a_1 + a_2c_1-c_1a_2, \, a_0c_4-c_4a_0 + a_1c_3-c_3a_1 + a_2c_2-c_2a_2, \\
&a_1c_4-c_4a_1 + a_2c_3-c_3a_2, \, a_2c_4-c_4a_2, \\
&b_0c_0-c_0b_0, \, b_0c_1-c_1b_0 + b_1c_0-c_0b_1, \, b_0c_2-c_2b_0 + b_1c_1-c_1b_1 + b_2c_0-c_0b_2, \\
&b_0c_3-c_3b_0 + b_1c_2-c_2b_1 + b_2c_1-c_1b_2, \, b_0c_4-c_4b_0 + b_1c_3-c_3b_1 + b_2c_2-c_2b_2, \\
&b_1c_4-c_4b_1 + b_2c_3-c_3b_2, \, b_2c_4-c_4b_2.
\end{split}
\] 
These terms give the relations in degree 2.

The higher order terms are determined by the following argument.
We have
\begin{equation}\label{1}
\begin{split}
&0 = xy-yx \\
&= (a_0z+a_1w+a_2t)(b_0z+b_1w+b_2t)-(b_0z+b_1w+b_2t)(a_0z+a_1w+a_2t) \\
&= (a_0b_0-b_0a_0)z^2 + (a_0b_1-b_1a_0+a_1b_0-b_0a_1)zw \\
&+ \{(a_0b_2-b_2a_0+a_2b_0-b_0a_2)zt + (a_1b_1-b_1a_1)w^2\} \\
&+ (a_1b_2-b_2a_1+a_2b_1-b_1a_2)wt + (a_2b_2-b_2a_2)t^2
\end{split}
\end{equation}
and
\begin{equation}\label{2}
\begin{split}
&0 = x(zt-w^2) - (zt-w^2)x \\
&=(a_0z+a_1w+a_2t)(zt-w^2+c_0z^2+c_1zw+c_2zt+c_3wt+c_4t^2) \\
&-(zt-w^2+c_0z^2+c_1zw+c_2zt+c_3wt+c_4t^2)(x+a_0z+a_1w+a_2t) \\
&=(a_0c_0-c_0a_0)z^3 +(a_0c_1-c_1a_0+a_1c_0-c_0a_1)z^2w \\
&+\{(a_0c_2-c_2a_0+a_2c_0-c_0a_2)z^2t + (a_1c_1-c_1a_1)zw^2\} \\
&+(a_0c_3-c_3a_0+a_1c_2-c_2a_1+a_2c_1-c_1a_2)zwt \\
&+\{(a_0c_4-c_4a_0+a_2c_2-c_2a_2)zt^2 + (a_1c_3-c_3a_1)w^2t\} \\
&+(a_1c_4-c_4a_1+a_2c_3-c_3a_2)wt^2+(a_2c_4-c_4a_2)t^3.
\end{split}
\end{equation}
Since
\[
\begin{split}
&(a_1b_1-b_1a_1)w^2 \equiv (a_1b_1-b_1a_1)(zt + c_0z^2+c_1zw+c_2zt+c_3wt+c_4t^2) \\
&(a_1c_1-c_1a_1)zw^2 \equiv (a_1c_1-c_1a_1)(z^2t + c_0z^3+c_1z^2w+c_2z^2t+c_3zwt+c_4zt^2) \\
&(a_1c_3-c_3a_1)w^2t \equiv (a_1c_3-c_3a_1)(zt^2 + c_0z^2t+c_1zwt+c_2zt^2+c_3wt^2+c_4t^3)
\end{split}
\]
modulo $(zt-w^2+c_0z^2+c_1zw+c_2zt+c_3wt+c_4t^2)$, (\ref{1}) and (\ref{2}) become
\[
\begin{split}
&(a_0b_0-b_0a_0+(a_1b_1-b_1a_1)c_0)z^2 \\
&+ (a_0b_1-b_1a_0+a_1b_0-b_0a_1+(a_1b_1-b_1a_1)c_1)zw \\
&+ (a_0b_2-b_2a_0+a_1b_1-b_1a_1+a_2b_0-b_0a_2 +(a_1b_1-b_1a_1)c_2)zt \\
&+ (a_1b_2-b_2a_1+a_2b_1-b_1a_2 + (a_1b_1-b_1a_1)c_3)wt \\
&+ (a_2b_2-b_2a_2 + (a_1b_1-b_1a_1)c_4)t^2
\end{split}
\]
and
\[
\begin{split}
&(a_0c_0-c_0a_0 + (a_1c_1-c_1a_1)c_0)z^3 \\
&+(a_0c_1-c_1a_0+a_1c_0-c_0a_1 +(a_1c_1-c_1a_1)c_1)z^2w \\
&+(a_0c_2-c_2a_0+a_1c_1-c_1a_1 + a_2c_0-c_0a_2 + (a_1c_1-c_1a_1)c_2 + (a_1c_3-c_3a_1)c_0)z^2t  \\
&+(a_0c_3-c_3a_0+a_1c_2-c_2a_1+a_2c_1-c_1a_2 + (a_1c_1-c_1a_1)c_3 + (a_1c_3-c_3a_1)c_1)zwt \\
&+(a_0c_4-c_4a_0+a_1c_3-c_3a_1 +a_2c_2-c_2a_2 +  (a_1c_1-c_1a_1)c_4 + (a_1c_3-c_3a_1)c_2)zt^2  \\
&+(a_1c_4-c_4a_1+a_2c_3-c_3a_2 + (a_1c_3-c_3a_1)c_3)wt^2 \\
&+(a_2c_4-c_4a_2 + (a_1c_3-c_3a_1)c_4)t^3.
\end{split}
\]
Therefore we have our claim.
\end{Expl}

\begin{Rem}\label{GS}
Let $C$ be a smooth rational curve on a Calabi-Yau $3$-fold.
If $C$ is contractible to a point by a bimeromorphic morphism $X \to \bar X$ 
whose exceptional locus coincides with $C$, then the
NC deformation ring of $\mathcal{O}_C$ is finite dimensional.
It is interesting to know whether the converse is true. 

By \cite{Clemens}, there is an example where $C$ is not contractible but the abelianization of 
the deformation ring is finite dimensional.
In this example, the normal bundle of $C$ is isomorphic to $\mathcal{O}(2) \oplus \mathcal{O}(-4)$ (hence not contractible).
The deformation ring is a quotient of a non-commutative formal power series ring with $3$ variables
by an ideal generated by $3$ relations.
By \cite{GS} and \cite{Vinberg}, 
it is known that such a ring is finite dimensional if the $3$ relations are generic quadratic forms
(this information, opposite to author's naive expectation, 
was given to the author by Professor Spela Spenko through 
Professors Michel Van den Bergh and Keiji Oguiso).
\end{Rem}

\begin{Rem}\label{T}
By \cite{Toda}~Lemma~4.1, the versal formal NC deformation is {\em convergent} in the sense that 
$\Vert m_n \Vert < C^n$ and
$\Vert f_n \Vert < C^n$ for suitable norms and a constant $C > 0$ which is independent of $n$. 
\end{Rem}

\begin{Rem}\label{ZH}
Zheng Hua informed the author that, if the bounded derived category of coherent sheaves $D^b(\text{coh}(X))$ 
has a strong exceptional collection consisting of line bundles, e.g., $X \cong \mathbf{P}^n$, 
then the NC deformation algebra of any coherent sheaf $F$ on $X$ is {\em algebraic} in the following sense; 
the $A^{\infty}$-algebra $\text{Ext}^*(F,F)$ is quasi-isomorphic to a 
finite dimensional $A^{\infty}$-algebra $B$ such that $m^B_n = 0$ for all $n \ge n_0$ with a fixed $n_0$
(cf. \cite{Hua}~Theorem~4.4, \cite{Hua2}). 
We note that $B$ is not necessarily minimal, i.e., $m^B_1$ may not vanish.
\end{Rem}

%%%%%%%%%%%%%%%%%%%%%%%%%%%%%%%%%%%%%%%%%%
%%%%%%%%%%%%%%%%%%%%%%%%%%%%%%%%%%%%%%%%%%
%%%%%%%%%%%%%%%%%%%%%%%%%%%%%%%%%%%%%%%%%%

Graduate School of Mathematical Sciences, University of Tokyo,
Komaba, Meguro, Tokyo, 153-8914, Japan. 

Morningside Center of Mathematics, 
Chinese Academy of Sciences, 
Haidian District, Beijing, China 100190

Department of Mathematical Sciences, 
Korea Advanced Institute of Science and Technology, 
291 Daehak-ro, Yuseong-gu, Daejeon 34141, Korea.

National Center for Theoretical Sciences, 
Mathematics Division, 
National Taiwan University, Taipei, 106, Taiwan.

kawamata@ms.u-tokyo.ac.jp

\end{document}